\documentclass{article}
\usepackage{graphicx,amsmath,amssymb,amsfonts,authblk}
\usepackage{amsthm}
\usepackage{paralist}
\usepackage{ifthen}
\usepackage{tikz}
\date{}

\newtheorem{theorem}{Theorem}[section]
\newtheorem{lemma}[theorem]{Lemma}
\newtheorem{remark}[theorem]{Remark}
\newtheorem{proposition}[theorem]{Proposition}
\newtheorem{corollary}[theorem]{Corollary}

\newcommand{\parder}[3][Default]{
	\frac{\partial \ifthenelse{\equal{#1}{Default}}{}{^{#1}}#2}{
              \partial #3 \ifthenelse{\equal{#1}{Default}}{}{^{#1}}}}

\numberwithin{equation}{section}

\makeatletter
\newcommand{\nolisttopbreak}{\vspace{\topsep}\nobreak\@afterheading}
\makeatother

\newcommand{\jac}{{\mathcal{J}}}
\newcommand{\hess}{{\mathcal{H}}}

\newcommand{\Mat}{\operatorname{Mat}}
\newcommand{\GL}{\operatorname{GL}}
\newcommand{\rk}{\operatorname{rk}}
\newcommand{\tr}{\operatorname{tr}}
\newcommand{\trdeg}{\operatorname{trdeg}}
\newcommand{\chr}{\operatorname{char}}
\newcommand{\diag}{\operatorname{diag}}
\newcommand{\tp}{^{\rm t}}

\begin{document}

\title{Quadratic homogeneous polynomial maps $H$ and Keller maps $x+H$ with
$\rk JH=3$}

\author{Michiel de Bondt\footnote{M.deBondt@math.ru.nl}}
\affil{\small Institute for Mathematics, Astrophysics and Particle Physics,
Radboud University Nijmegen, The Netherlands}
\author{Xiaosong Sun\footnote{Corresponding author, E-mail: sunxs@jlu.edu.cn}}
\affil{\small School of Mathematics, Jilin University, Changchun 130012, China}

\maketitle

\begin{abstract}
We classify all quadratic homogeneous polynomial maps $H$ and
Keller maps of the form $x + H$, for which $\rk \jac H = 3$, over a field $K$
of arbitrary characteristic. In particular, we show that such a Keller map (up to a square part if $\chr K=2$) is a tame automorphism.
\end{abstract}

\section{Introduction}

Throughout the paper, $K$ is a field and $K[x]:=K[x_1,x_2,\ldots,x_n]$ stands for the polynomial ring in $n$ variables.
For a polynomial map $F=(F_1,F_2,\ldots,F_m)\in K[x]^m$, we denote by
$\jac F:=(\frac{\partial F_i}{\partial x_j})_{m\times n}$ the Jacobian matrix of $F$ and $\deg F:=\max_i \deg F_i$ the degree of $F$.
Write $F\circ G$ or $FG$ for the composition of two polynomial maps. By the chain rule, $\jac (F\circ G)=(\jac F)|_{x=G}\cdot \jac G$.
 A polynomial map $H\in K[x]^m$ is called homogeneous of degree $d$ if each $H_i$ is zero or homogeneous of degree $d$.

A polynomial map $F\in K[x]^n$ is called a Keller map if $\det \jac F\in K^*$. The Jacobian conjecture asserts that,  when $\chr K=0$, a Keller map is invertible; see \cite{essen2000} or \cite{bass1982}.
It is still open for any dimension $n\geq 2$.

Bass et al.\@ \cite{bass1982} showed that it suffices to consider the Jacobian conjecture for all cubic homogeneous polynomial maps.
Wang \cite{wang} showed that, when $\chr K=0$, any quadratic Keller map is invertible, however little is known for the structure of them.

A polynomial automorphism of the form $E_{i,a}:=(x_1,\ldots,x_{i-1},x_i+a,x_{i+1},\ldots,\\x_n)$ is called elementary,
where $a\in K[x]$ contains no $x_i$. A polynomial
automorphism is called tame if it is a finite composition of
elementary ones and affine ones (i.e. those of degree 1).

The Tame Generators Problem asks if every
polynomial automorphism is tame. It has an affirmative
answer in dimension 2 for arbitrary characteristic (see \cite{jung, kulk}) and
a negative answer in dimension $3$ for the case of $\chr K=0$ (see \cite{shestakov2}), and is still
open for any dimension $n\geq 4$.

Rusek \cite{rusek} conjectured that every quadratic polynomial
automorphism is tame, see also \cite[Section 5.2]{essen2000}. It is the quadratic case of the Tame Generators Problem.

Meisters and Olech \cite{meisters1991} showed that, when $\chr K=0$, Rusek's conjecture has an affirmative answer in dimension $n\leq 4$. De Bondt \cite{bondt09} and Sun \cite{sun} independently showed that, when $\chr K=0$, Rusek's conjecture has an affirmative for $n=5$. And Rusek's conjecture is
still open for any dimension $n\geq 6$.

Sun \cite{sun10} showed that, when $\chr K=0$, any quadratic homogeneous quasi-translation is tame in dimension $n\leq 9$. Recently,
de Bondt \cite{bondt17} classified all quadratic polynomial maps $H$ with $\rk \jac H=2$ for any characteristic, and showed that if $\jac H$ is nilpotent then $\jac H$ is similar over $K$ to a triangular one.

In this paper, we investigate quadratic homogeneous polynomial maps $H$ with $\rk \jac H=r\geq 3$ for any characteristic. In Section 2, we obtain some general results for any $r$ (Theorem \ref{rkr}). And in subsequent sections, we focus on the case of $r=3$. In Section 3, we classify all quadratic homogeneous polynomial maps $H$ with $\rk \jac H=3$ (Theorem \ref{rk3}), and in Section 4, we classify the corresponding Keller maps $x+H$, and show that such a Keller map (up to a square part if $\chr K=2$) is a tame automorphism (Theorem \ref{rk3np}).

\section{Quadratic homogeneous polynomial maps $H$ with $\rk \jac  H=r$}

In this section, we are devoted to obtaining some general results on the structure of quadratic homogeneous polynomial maps $H$ with
$\rk \jac H=r$ for any characteristic. The main result is the following Theorem \ref{rkr}.

\begin{theorem} \label{rkr}
Let $H \in K[x]^m$ be a quadratic homogeneous polynomial map, and
$r := \rk \jac H$. Then there are $S \in \GL_m(K)$ and $T \in \GL_n(K)$, such that
for $\tilde{H} := S H(Tx)$, only the first $\frac{1}{2} r^2 + \frac12 r$ rows
of $\jac \tilde{H}$ may be nonzero, and one of the following statements holds:
\begin{enumerate}[(1)]

\item Only the first $\frac{1}{2} r^2 - \frac{1}{2} r + 1$ rows of $\jac \tilde{H}$ may be nonzero;

\item $\chr K \ne 2$ and only the first $r$ columns of $\jac \tilde{H}$ are nonzero;

\item $\chr K = 2$ and only the first $r+1$ columns of $\jac \tilde{H}$ are nonzero.

\end{enumerate}
Conversely, $\rk \jac \tilde{H} \le r$ if either $\tilde{H}$ is as in
(2) or (3), or $1 \le r \le 2$ and $\tilde{H}$ is as in
(1).
\end{theorem}

To prove Theorem \ref{rkr}, we start with some lemmas.

\begin{lemma} \label{Fqlem}
Let $L$ be an extension field of $K$. If Theorem \ref{rkr}
holds for $L$, then it holds for $K$.
\end{lemma}

\begin{proof}
We only prove this lemma for the first claim of Theorem \ref{rkr},
because the second claim can be treated in a similar manner, and the last
claim does not depend on the  base field.

Suppose that $H\in K[x]^m$ satisfies the first claim of Theorem \ref{rkr} for $L$, i.e., there are $S \in \GL_m(L)$ and $T \in \GL_n(L)$ such that
for $\tilde{H} := S H(Tx)$, only the first $\frac12 r^2 + \frac12 r$ rows
of $\jac \tilde{H}$ may be nonzero. The first claim of Theorem \ref{rkr} holds for $K$ obviously if $m \leq \frac{1}{2} r^2 + \frac12 r$. So assume that $m > \frac12 r^2 + \frac12 r$. Then the rows of $\jac H$ are
dependent over $L$. Since $L$ is a vector space over $K$,
the rows of $\jac H$ are dependent over $K$. So we
may assume that the last row of $\jac H$ is zero. By induction on $m$,
$(H_1,H_2,\ldots,H_{m-1})$ satisfies the first claim of Theorem \ref{rkr} for $K$ and thus $H$ satisfies the first claim of Theorem \ref{rkr} for $K$.
\end{proof}

\begin{lemma} \label{23}
If a quadratic homogeneous polynomial $\tilde{H}$ is of the form in (2) or (3) of
Theorem \ref{rkr}, then $\rk \jac \tilde{H}\leq r$ and there exists an $S\in \GL_m(K)$ such that only the first $\frac12 r^2 + \frac12 r$ rows
of $\jac (S\tilde{H})$ may be nonzero.
\end{lemma}

\begin{proof} If $\tilde{H}$ is as in (2) of Theorem \ref{rkr}, then it is obvious that $\rk \jac \tilde{H} \le r$. Notice that in this case $\tilde{H}$ contains only terms in $x_1,x_2,\ldots,x_r$.  Since
the number of quadratic terms in $x_1, x_2, \ldots, x_r$ is $\binom{r+1}{2}=\frac12 r^2 + \frac12 r$,  the conclusion follows.

If $\tilde{H}$ is as in
(3) of Theorem \ref{rkr}, then $\rk \jac \tilde{H} \le r$ as well since $\jac \tilde{H} \cdot x = 2\tilde{H}=0$ when $\chr K = 2$.
Notice that in this case all the square-free terms of $\tilde{H}$ are in $x_1,x_2,\ldots,x_r,x_{r+1}$.  Since the number of square-free quadratic terms in $x_1, x_2, \ldots, x_r,x_{r+1}$ is also $\binom{r+1}{2}=\frac{1}{2} r^2 + \frac12 r$, the conclusion follows.
\end{proof}

\begin{lemma} \label{Irlem}
Let $M$ be a nonzero matrix whose entries are linear forms in $K[x]$. Suppose
that $r := \rk M$ does not exceed the cardinality of $K$.
Then there are invertible matrices $S$ and $T$ over $K$, such that for $\tilde{M} := S M T$,
$$
\tilde{M} = \tilde{M}^{(1)} L_1 + \tilde{M}^{(2)} L_2 + \cdots + \tilde{M}^{(n)} L_n
$$
where $\tilde{M}^{(i)}$ is a matrix with coefficients in $K$ for each $i$,
$L_1, L_2, \ldots, L_n$ are independent linear forms, and
$
\tilde{M}^{(1)} = \left( \begin{array}{cc}
I_r & 0 \\ 0 & 0 \\
\end{array} \right).
$
\end{lemma}

\begin{proof}
We may assume without loss of generality that the determinant $f:=\det M_0$ is nonzero, where $M_0$ is the
principal submatrix of size $r \times r$ of $M$. Since $f$ is a homogeneous polynomial of degree $r$, we deduce from \cite[Lemma 5.1 (ii)]{bondt13} that
there exists a $v \in K^n$ such that $f(v) \ne 0$.

Take independent linear forms $L_1, L_2, \ldots, L_n$ such that $L_i(v) = 0$
for all $i \ge 2$. Then $L_1(v) \ne 0$, and we may assume that $L_1(v) = 1$.
Write
$$
M = M^{(1)} L_1 + M^{(2)} L_2 + \cdots + M^{(n)} L_n,
$$
where each $M^{(i)}$ is a matrix over $K$.
Since $M^{(1)}=M(v)$, we have $\rk M^{(1)}=r$ and its leading principal minor of size $r$ is nonzero, and thus we may choose invertible matrices $S$ and $T$ over $K$,
such that
$
S M^{(1)} T = \left( \begin{array}{cc}
I_r & 0 \\ 0 & 0 \\
\end{array} \right).
$
Finally, take $\tilde{M}=SMT$ and $\tilde{M}^{(i)} = S M^{(i)} T$ for each $i$.
\end{proof}

Suppose that $\tilde{M}$ is as in Lemma \ref{Irlem}. Write
\begin{equation} \label{ABCD}
\tilde{M} = \left( \begin{array}{cc} A & B \\ C & D \end{array} \right)
\end{equation}
where $A \in \Mat_r(K[x])$. If we extend $A$ with one row and one
column of $\tilde{M}$, we get an element of $\Mat_{r+1}\big(K[x]\big)$ whose
determinant is zero. If we focus on the coefficients of $L_1^r$
and $L_1^{r-1}$ of this determinant, we see that
\begin{equation} \label{D0CB0}
D = 0 \qquad \mbox{and} \qquad CB = 0.
\end{equation}

\begin{lemma} \label{Ccoldep1}
Let $\tilde{H} \in K[x]^m$, such that $\jac \tilde{H}$ is as $\tilde{M}$ in
Lemma \ref{Irlem} and write $\tilde{M} = \left( \begin{array}{cc} A & B \\ C & D \end{array} \right)$ as in \eqref{ABCD}. Suppose that $\chr K \ne 2$. Then
\begin{enumerate}[(i)]

\item If $C \ne 0$, then there exists a $v \in K^n$ of which the first $r$
coordinates are not all zero, such that
$
(\jac \tilde{H}) \cdot v = \left( \begin{array}{cc} I_r & 0 \\
0 & 0 \end{array} \right) \cdot x.
$

\item The columns of $C$ are dependent over $K$.
\end{enumerate}
\end{lemma}

\begin{proof}
(\emph{i}) Take $v$ as in Lemma \ref{Irlem}, and write $v = (v',v'')$,
such that $v' \in K^r$ and $v'' \in K^{n-r}$. Since $\tilde{H}$ is quadratic
homogeneous, we have
$$
(\jac \tilde{H}) \cdot v = (\jac \tilde{H})|_{x=v} \cdot x = \tilde{M}^{(1)} \cdot x
= \left( \begin{array}{cc} I_r & 0 \\ 0 & 0 \end{array} \right) \cdot x.
$$

From $CB = 0$, we deduce that
$$
CA\,v' = CA\,v' + CB\,v'' = C\,(A|B)\,v
= C \left( \begin{smallmatrix} x_1 \\ x_2 \\[-5pt] \vdots \\ x_r \end{smallmatrix} \right)
= 2 \left( \begin{smallmatrix} \tilde{H}_{r+1} \\ \tilde{H}_{r+2} \\[-5pt] \vdots \\
    \tilde{H}_{m} \end{smallmatrix} \right).
$$
Since $C\neq 0$, we have $\tilde{H}_{i} \ne 0$ for some $i > r$, and thus the right-hand side is nonzero. Therefore $v' \ne 0$ and the conclusion (\emph{i}) follows.

(\emph{ii}) We may assume that $C\neq 0$. Take $v'$ as in (i). From $D=0$, we deduce that $C\cdot v'=(C|D)\cdot v=0$, which yields (\emph{ii}).
\end{proof}

\begin{lemma} \label{BCrkr}
Use the same notations as in Lemma \ref{Ccoldep1}. Suppose that $\rk B + \rk C = r$ and that the columns of $C$ are dependent
over $K$. Then the column space of $B$ (over $K(x)$) contains a nonzero constant vector (over $K$).
\end{lemma}

\begin{proof}
From $\rk C + \rk B = r$ and $CB = 0$, we deduce that $\ker C$ is equal
to the column space of $B$. Hence any $w \in K^r$
such that $Cw = 0$ is contained in the column space of $B$.
\end{proof}

Now we are in the position to prove Theorem \ref{rkr}.

\begin{proof}[Proof of Theorem \ref{rkr}]
By Lemma \ref{Fqlem}, we may assume that $K$ has at least
$r$ elements. Let $M = \jac H$ and take $S$ and $T$ as in Lemma \ref{Irlem}.
Then $S (\jac H) T$ is as $\tilde{M}$ in Lemma \ref{Irlem}. Let $\tilde{H} := SH(Tx)$.
Then $\jac \tilde{H} = S (\jac H) |_{x=Tx} T$ is as $\tilde{M}$ in Lemma \ref{Irlem}
as well, up to replacing $L_i$ by $L_i(Tx)$.

Take $\tilde{M} = \jac \tilde{H}$ and take $A$, $B$, $C$, $D$ as in \eqref{ABCD}.
We distinguish four cases:
\begin{itemize}

\item \emph{The column space of $B$ contains a nonzero constant vector.}

Then there exists an $U \in \GL_m(K)$, such that the column space of
$U \tilde{M}$ contains $e_1$, because $D = 0$. Consequently, the matrix which
consists of the last $m-1$ rows of $\jac (U \tilde{H}) = U \tilde{M}$ has rank $r-1$.
By induction on $r$, it follows that we may choose $U$ such that
only
$$
\tfrac12(r-1)^2 + \tfrac12(r-1) = \tfrac12 r^2 - \tfrac12 r
$$
rows of $\jac (U \tilde{H})$ may be nonzero besides the first row of
$\jac (U \tilde{H})$. So $U \tilde{H}$ is as $\tilde{H}$ in (1) of
Theorem \ref{rkr}.

\item \emph{The rows of $B$ are dependent over $K$ in pairs.}

If $B \ne 0$, then the column space of $B$ contains a nonzero
constant vector, and the case above applies since $D=0$.

So assume that $B = 0$. Then only the first $r$ columns of
$\jac \tilde{H}$ may be nonzero. Since
$\rk \jac \tilde{H} = r$, the first $r$ columns of $\jac \tilde{H}$
are indeed nonzero.
Furthermore, it follows from $\jac \tilde{H} \cdot x = 2 \tilde{H}$
that $\chr K \neq 2$. So $\tilde{H}$ is as in (2) of Theorem \ref{rkr},
and the result follows from Lemma \ref{23}.

\item \emph{$\chr K = 2$ and $\rk B \le 1$.}

If the rows of $B$ are dependent over $K$ in pairs, then the second case
above applies, so assume the converse. Then on account of \cite[Theorem 2.1]{bondt17},
the columns of $B$ are dependent over $K$ in pairs. As $D = 0$, there exists an
$U'' \in \GL_{n-r}(K)$ such that only the first column
of
$
\binom{B}{D} U''
$
may be nonzero. Hence there exists an $U \in \GL_n(K)$
such that only the first $r+1$ columns of $(\jac \tilde{H})\, U$
may be nonzero. Consequently, $\tilde{H}(Ux)$ is as
$\tilde{H}$ in (3) of Theorem \ref{rkr}, and the result follows from Lemma \ref{23}.

\item \emph{None of the above.}

We first show that $\rk C \le r - 2$. So assume that $\rk C \ge r - 1$. Since $CB=0$, we have
$\rk C + \rk B \le r$, and thus $\rk B \le 1$. As the last case
above does not apply, $\chr K \ne 2$. By Lemma \ref{Ccoldep1}, the columns of $C$ are dependent over $K$. As the first case above does not
apply, it follows from Lemma \ref{BCrkr} that $\rk C + \rk B < r$.
So $B = 0$, which is the second case above, a contradiction.
So $\rk C \le r - 2$ indeed.

By induction on $r$, we may assume that $C$ has at most
$$
\tfrac12 (r-2)^2 + \tfrac12 (r-2) = \tfrac12 r^2 - \tfrac32 r + 1
$$
nonzero rows. As $A$ has $r$ rows, there exists an $U \in \GL_m(K)$ such that
$U\tilde{H}$ is as $\tilde{H}$ in (1) of Theorem \ref{rkr}.

\end{itemize}
The last claim of Theorem \ref{rkr} follows from Lemma \ref{23}
and the fact that $\frac12 r^2 - \frac12 r + 1 = r$ if $1 \le r \le 2$.
\end{proof}

\section{Quadratic homogeneous polynomial maps $H$ with $\rk \jac  H=3$}

In this section, we classify all quadratic homogeneous polynomial maps $H$ with $\rk \jac H=3$ for any characteristic.

\begin{theorem} \label{rk3}
Let $H \in K[x]^m$ be a quadratic homogeneous polynomial map with $\rk \jac H = 3$. Then there are $S \in \GL_m(K)$ and $T \in \GL_n(K)$,
such that for $\tilde{H} := S H(Tx)$, one of the following statements holds:
\begin{enumerate}[(1)]

\item Only the first $3$ rows of $\jac \tilde{H}$ may be nonzero;

\item Only the first $4$ rows of $\jac \tilde{H}$ may be nonzero, and
$$
(\tilde{H}_1,\tilde{H}_2,\tilde{H}_3,\tilde{H}_4) =
(\tilde{H}_1, \tfrac12 x_1^2, x_1 x_2,  \tfrac12 x_2^2)
$$
(in particular, $\chr K \ne 2$);

\item Only the first $4$ rows of $\jac \tilde{H}$ may be nonzero,
$$
\jac (\tilde{H}_1,\tilde{H}_2,\tilde{H}_3,\tilde{H}_4) =
\jac (\tilde{H}_1, x_1 x_2, x_1 x_3, x_2 x_3)
$$
and $\chr K = 2$;

\item Only the first $4$ rows of $\jac \tilde{H}$ may be nonzero, and
$$
\big(\tilde{H}_1,\tilde{H}_2,\tilde{H}_3,\tilde{H}_4\big) =
\big( x_1 x_3 + c x_2 x_4, x_2 x_3 - x_1 x_4,
\tfrac12 x_3^2 + \tfrac{c}2 x_4^2, \tfrac12 x_1^2 + \tfrac{c}2 x_2^2 \big)
$$
for some nonzero $c \in K$ (in particular, $\chr K \ne 2$).

\item $\chr K\neq 2$ and only the first $3$ columns of $\jac \tilde{H}$ are nonzero.

\item $\chr K= 2$ and only the first $4$ columns of $\jac \tilde{H}$ are nonzero.

\end{enumerate}
Conversely, $\rk \jac \tilde{H} \le 3$ in each of the five statements above.
\end{theorem}

\begin{corollary} \label{rktrdeg}
Let $H \in K[x]^m$ be quadratic homogeneous  such that
$\rk \jac H \le 3$. If $\chr K \neq 2$, then $\rk \jac H = \trdeg_K K(H)$.
\end{corollary}

\begin{proof}
Since $\rk \jac H \le \trdeg_K K(H)$, it suffices to show that
$\trdeg_K K(H) \le 3$ if $\chr K \neq 2$.
In (4) of Theorem \ref{rk3}, we have $\trdeg_K K(H) \le 3$ because
$$
\tilde{H}_1^2 + c \tilde{H}_2^2 - 4 \tilde{H}_3 \tilde{H}_4 = 0.
$$
In the other cases of Theorem \ref{rk3} where $\chr K \neq 2$,
$\trdeg_K K(H) \le 3$ follows trivially.
\end{proof}

\begin{lemma} \label{Ccoldep2}
Let $\tilde{H} \in K[x]^m$, such that $\jac \tilde{H}$ is as $\tilde{M}$ in
Lemma \ref{Irlem}, and write $\tilde{M} = \left( \begin{array}{cc} A & B \\ C & D \end{array} \right)$
as in \eqref{ABCD}. If $\rk C = 1$ and $r$ is odd, then the columns of $C$ are dependent over $K$.
\end{lemma}

\begin{proof}
The case where $\chr K \ne 2$ follows from Lemma \ref{Ccoldep1}, so
assume that $\chr K = 2$.
Since $\rk C = 1 = \frac12 \cdot 1^2 + \frac12 \cdot 1$, we deduce from
Theorem \ref{rkr} that the rows of $C$ are dependent over $K$ in pairs.
Say that the first row of $C$ is nonzero.

For any $f\in K[x]$, we denote by $\hess (f):=(\frac{\partial^2 f}{\partial x_i\partial x_j})_{n\times n}$ the Hessian matrix of $f$.
As $r$ is odd, it follows from Proposition \ref{evenrk} and Remark \ref{evenrk1} below that
$\rk \hess \tilde{H}_{r+1} < r$. Hence there exists a $w \in K^r$
such that $(\hess \tilde{H}_{r+1}) \, w = 0$, and thus
$$
(\jac \tilde{H}_{r+1}) \, w = x\tp (\hess \tilde{H}_{r+1})w \, = 0.
$$
Since the row space of $C$ is spanned by $\jac \tilde{H}_{r+1}$, we have $C\,w = 0$.
\end{proof}

It is well-known that, if $\chr K\neq 2$ and $M \in \Mat_{n}(K)$ is symmetric matrix, then there exists a $T \in \GL_{n}(K)$, such that $T\tp M T$ is a diagonal matrix.

\begin{proposition} \label{evenrk}
Let $M \in \Mat_{n}(K)$ be a symmetric matrix with zeroes on the diagonal. Then $\rk M$ is even, and
there exists a lower triangular matrix $T \in \Mat_{n}(K)$ with
ones on the diagonal, such that $T\tp M T$ is the product of a symmetric
permutation matrix and a diagonal matrix.
\end{proposition}

\begin{proof} If the last column of $M$ is zero, then we have reduced the problem to the leading
principal submatrix of size $n-1$.  And if the last column of $M$ is not zero, let $i$ be the index of the lowest nonzero entry in the last column of $M$, and use $M_{in}$ and $M_{ni}$ as pivots to clean the rest elements of columns $i$ and $n$ and rows $i$ and $n$ of $M$ to obtain a matrix $\hat{M}$, and then we reduce the problem to the submatrix obtained by removing the last row and last column of $\hat{M}$. The conclusion follows by induction on $n$.
\end{proof}

\begin{remark} \label{evenrk1}  When $\chr K=2$, the Hessian
matrix of a quadratic homogeneous polynomial is symmetric with zeroes on the diagonal,
and thus is of even rank.
\end{remark}

\begin{lemma} \label{rk3calc}
Let $H \in K[x_1,x_2,x_3,x_4]^4$ be quadratic homogeneous with
$$
\jac H_4 = (\, x_1 ~ c x_2 ~ 0 ~ 0 \,)
\qquad \mbox{and} \qquad
\jac H \cdot v = \left( \begin{smallmatrix}
x_1 \\ x_2 \\ x_3 \\ 0 \end{smallmatrix} \right)
$$
for some nonzero $c \in K$, and a $v \in K^4$ of which
the first $3$ coordinates are not all zero. And suppose that the last column of
$\jac H$ does not generate a nonzero constant vector. Then there are $S, T \in \GL_4(K)$, such that
$$SH(Tx)=\big(x_1 x_3 + c x_2 x_4, x_2 x_3 - x_1 x_4, \tfrac12 x_3^2 + \tfrac{c}2 x_4^2, \tfrac12 x_1^2 + \tfrac{c}2 x_2^2\big).$$
\end{lemma}

\begin{proof} Noticing that $\frac{\partial H_4}{\partial x_1}=x_1$, we have $\chr K\neq 2$. Since the last row of $\jac H$ is $(\, x_1 ~ c x_2  ~ 0 ~ 0 \,)$
and the last coordinate of $\jac \tilde{H} \cdot v$ is zero, we deduce that
$v_1 = v_2 = 0$. As the first $3$ coordinates of $v$ are not all zero,
we have $v_3 \ne 0$.

Let $S=\diag(v_3,v_3,v_3,1)$, $T=(e_1,e_2,v_3^{-1}v,ke_4)$ where $k$ is any nonzero constant, and let $\tilde{H}=SH(Tx)$.
Then
\begin{align*}
(\jac \tilde{H}) \cdot e_3
&= S\, (\jac H)|_{x=Tx} \cdot T e_3
 = \big(S\,(\jac H) \cdot v_3^{-1} v\big)\big|_{x=Tx} \\
&= v_3^{-1} S\, \left.
  \left( \begin{smallmatrix} x_1 \\ x_2 \\ x_3 \\ 0 \end{smallmatrix} \right)
  \right|_{x=Tx}
 = \left( \begin{smallmatrix} x_1 \\ x_2 \\ x_3 \\ 0 \end{smallmatrix} \right)
\end{align*}
and
$
\tilde{H}_4 = H_4(Tx) = H_4.
$
Write $\jac \tilde{H} = M^{(1)} x_3 + M^{(2)} x_2 +  M^{(3)} x_1+M^{(4)} x_4.$
Then $$\left(\begin{smallmatrix} x_1 \\ x_2 \\ x_3 \\ 0 \end{smallmatrix} \right)=\jac \tilde{H}\cdot e_3=(\jac \tilde{H})|_{x=e_3}\cdot x=M^{(1)}\cdot x$$ and thus
$M^{(1)}=\diag(1,1,1,0)$.
It follows that
$\jac \tilde{H}$ is as $\tilde{M}$ in Lemma \ref{Irlem} with
$L_1 = x_3$, $L_2 = x_2$, $L_3 = x_1$ and $L_4 = x_4$. Write $\tilde{M} = \left( \begin{array}{cc} A & B \\ C & D \end{array} \right)$ as in \eqref{ABCD}. Then $C=(x_1,cx_2,0)$.

Just like the last column of $\jac H$, the last column of $\jac \tilde{H}$ does
not generate a nonzero constant vector. So $B_{11}$ and $B_{21}$ are not both
zero. Then by $CB=0$ we deduce that $B=(cx_2,-x_1,B_{31})\tp$ up to a scalar, and the scalar can be chosen to be 1 by adapting the value of $k$ in $T$.

The coefficient of $x_3$ in $B_{31}$ is zero, and by changing the third row
of $I_4$ on the left of the diagonal in a proper way, we can get an
$U \in \GL_4(K)$ such that
$$
U \binom{B}{D} = \left( \begin{smallmatrix}
c x_2 \\ - x_1 \\ \tilde{c} x_4 \\ 0
\end{smallmatrix} \right)
$$
for some $\tilde{c} \in K$. Since $U^{-1}$ can be obtained
by changing the third row of $I_4$ on the left of the diagonal
in a proper way as well, we infer that
\begin{align*}
\jac \big(U\tilde{H} (U^{-1}x)\big) \cdot e_4
&= U\, (\jac \tilde{H})|_{x = U^{-1}x} \,U^{-1} \cdot e_4
 = U\, (\jac \tilde{H})|_{x = U^{-1}x} \cdot e_4 \\
&= U\binom{B}{D}\bigg|_{x = U^{-1}x}=U\binom{B}{D}= \left( \begin{smallmatrix}
c x_2 \\ - x_1 \\ \tilde{c} x_4 \\ 0
\end{smallmatrix} \right).
\end{align*}
Similarly one may verify $\jac \big(U\tilde{H} (U^{-1}x)\big) \cdot e_3
 = \left( \begin{smallmatrix} x_1 \\ x_2 \\ x_3 \\ 0 \end{smallmatrix} \right)$ and
$
U_4\, \tilde{H} (U^{-1}x) = \tilde{H}_4 (U^{-1}x) = \tilde{H}_4.
$
So $\jac\big(U^{-1} \tilde{H}(Ux)\big)$ is of the form $$\left( \begin{array}{cccc}
A_{11} & A_{12} & x_1 & c x_2 \\
A_{21} & A_{22} & x_2 & -x_1 \\
A_{31} & A_{32} & x_3 & \tilde{c}x_4 \\
x_1 & c x_2 & 0 & 0
\end{array} \right)
$$
where $c,\tilde{c} \in K$, such that $c \neq 0$.

By row operations using $C_{11}=x_1$ as a pivot,  we may also assume that the coefficients of
$x_1$ in $A_{11}$, $A_{21}$, $A_{31}$ equal to zero.

Replacing $\tilde{H}$ by $U^{-1} \tilde{H}(Ux)$. Let $a_i$ and $b_i$ be the coefficients of $x_1$ and $x_2$ in
$A_{i2}$ respectively, for each $i \le 3$. Then
\begin{gather*}
\tilde{H} = \left( \begin{array}{c}
a_1 x_1 x_2 + \frac12 b_1 x_2^2 + x_1 x_3 + c x_2 x_4 \\
a_2 x_1 x_2 + \frac12 b_2 x_2^2 + x_2 x_3 - x_1 x_4 \\
a_3 x_1 x_2 + \frac12 b_3 x_2^2 + \frac12 x_3^2 + \frac{\tilde{c}}2 x_4^2 \\
\frac12 x_1^2 + \frac{c}2 x_2^2
\end{array} \right)
\qquad \mbox{and} \\
\jac \tilde{H} = \left( \begin{array}{cccc}
a_1 x_2 + x_3 & a_1 x_1 + b_1 x_2 + c x_4 & x_1 & c x_2 \\
a_2 x_2 - x_4 & a_2 x_1 + b_2 x_2 + x_3 & x_2 & -x_1 \\
a_3 x_2 & a_3 x_1 + b_3 x_2 & x_3 & \tilde{c} x_4 \\
x_1 & c x_2 & 0 & 0
\end{array} \right).
\end{gather*}
Consequently, it suffices to show that $a_i = b_i = 0$ for each $i \le 3$,
and that $\tilde{c} = c$.

By assumption, $\det \jac H=0$.
Observing the coefficient of $x_1^4$ in $\det \jac \tilde{H}$ by expanding $\jac H$ along rows
4, 3, 2, 1, in that order, we see that
$a_3=0$. Hence the third row of $\jac \tilde{H}$ reads
$
\jac \tilde{H}_3 = (\, 0 ~ b_3 x_2 ~ x_3 ~ \tilde{c} x_4 \,).
$
Since the coefficients of $x_1^3 x_2$ and $x_1^3 x_3$ in $\det \jac \tilde{H}$
are zero, we see by expanding along rows $3$, $4$, $1$, in that order, that
$b_3 x_2 = a_1 x_1 = 0$. Hence the third row of $\jac \tilde{H}$ reads
$$
\jac \tilde{H}_3 = (\, 0 ~ 0 ~ x_3 ~ \tilde{c} x_4 \,).
$$
Since the coefficient of $x_2^3 x_3$ in $\det \jac \tilde{H}$
is zero, we see by expanding along rows $3$, $4$, $2$, in that order, that
$a_2 x_2 = 0$. So
$$
\jac \tilde{H} = \left( \begin{array}{cccc}
x_3 & b_1 x_2 + c x_4 & x_1 & c x_2 \\
- x_4 & b_2 x_2 + x_3 & x_2 & -x_1 \\
0 & 0 & x_3 & \tilde{c} x_4 \\
x_1 & c x_2 & 0 & 0
\end{array} \right).
$$
Since the coefficient of $x_1^2 x_3 x_4$ in $\det \jac \tilde{H}$
is zero, we see by expending along row $3$, and columns $2$ and $1$,
in that order, that $\tilde{c} x_4 = c x_4$.

Since the coefficient of $x_1 x_2^2 x_3$ in $\det \jac \tilde{H}$
is zero, we see by expending along row $3$, and columns $1$, $4$, in that order,
that $b_2 x_2 = 0$. Using that and that the coefficient of $x_1^2 x_2 x_3$ in
$\det \jac \tilde{H}$ is zero, we see by expending along row $3$, and columns
$1$, $2$, in that order, that $b_1 x_2 = 0$.

In conclusion, $a_i = b_i = 0$ for each $i \le 3$,
and $\tilde{c} = c$, and thus $\tilde{H}$ is as claimed.
\end{proof}

Now we can prove Theorem \ref{rk3}.

\begin{proof}[Proof of Theorem \ref{rk3}]
From Lemma \ref{F2lem} below, it follows that we may assume that $K$ has at least
$3$ elements. Hence we may assume that $\tilde{M}:= \jac \tilde{H}$ is as in Lemma
\ref{Irlem} and write
$\tilde{M} = \left( \begin{array}{cc} A & B \\ C & D \end{array} \right)$ as in \eqref{ABCD}.
We distinguish three cases:
\begin{itemize}

\item \emph{The column space of $B$ contains a nonzero constant vector.}

Then there exists an $U \in \GL_m(K)$, such that the column space of
$U \tilde{M}$ contains $e_1$. So the matrix which consists of the last
$m-1$ rows of $\jac (U \tilde{H}) = U \tilde{M}$ has rank $2$.
Let  $\tilde{U}$ be the matrix consisting of the last $m-1$ rows of $U$.
Then $\rk \jac (\tilde{U} \tilde{H}) = 2$, and we apply Theorem
\ref{rkr} to $\tilde{U} \tilde{H}$.
\begin{compactitem}

\item If case (1) of Theorem \ref{rkr} applies for $\tilde{U} \tilde{H}$,
then we may assume that only the first $\tfrac12\cdot2^2 - \tfrac12\cdot 2 + 1 = 2$
rows of $\tilde{U} \tilde{H}$ may be nonzero, and thus only the first $3$ rows of
$U\tilde{H}$ may be nonzero. So case (1) of Theorem \ref{rk3} follows.

\item If case (2) of Theorem \ref{rkr} applies for $\tilde{U} \tilde{H}$,
then $\chr K\neq 2$ and only the first $2$ columns of $\jac (\tilde{U} \tilde{H})$
are nonzero, and thus case (1) or case (2) of Theorem \ref{rk3} follows.

\item If case (3) of Theorem \ref{rkr} applies for $\tilde{U} \tilde{H}$,
then $\chr K= 2$ and only the first $3$ columns of $\jac (\tilde{U} \tilde{H})$
are nonzero, and thus case (1) or case (3) of Theorem \ref{rk3} follows.

\end{compactitem}

\item \emph{The columns of $B$ are dependent over $K$ in pairs.}

We may assume that $C\neq 0, B\neq 0$ and $\chr K\neq 2$. (In fact, if $C = 0$, then (1) of Theorem \ref{rk3}
follows; if $B = 0$, then $\tilde{H}$ is as in (2) of Theorem \ref{rkr}, which is
(5) of Theorem \ref{rk3}; if $\chr K = 2$, then $\tilde{H}$ is as in (3) of Theorem \ref{rkr}, which
is (6) of Theorem \ref{rk3}.)

Since $\chr K\neq 2$, it follows from Lemma \ref{Ccoldep1} that the columns of $C$ are
dependent over $K$, and thus $\rk C \leq 2$. Notice that $\rk B=1$. If $\rk C = 2$, then
$\rk B + \rk C = 3$. By  Lemma \ref{BCrkr}, $B$ contains a nonzero constant
vector, and thus (1) of Theorem \ref{rk3}
follows.

So $\rk C = 1$. By
Theorem \ref{rkr}, we may assume that only the first row of $C$ is nonzero. We may also assume that only the first column of $B$ is nonzero. By Lemma \ref{Ccoldep2} the
columns of $C$ are dependent over $K$.

By coordinate change, we may assume that $\tilde{H}_{4}=a_1x_1^2+\frac{c}{2}x_2^2+a_3x_3^2$, and we may assume that
$a_1=\frac12$ and $a_3=0$, since $C\neq 0$ and the columns of $C$ are dependent over $K$.

Then the first row of $C$ is $(\, x_1 ~ c x_2  ~ 0 \,)$.
We distinguish two cases.
\begin{compactitem}

\item $c = 0$.

Noticing that the first column of $\jac \tilde{H}$ is independent of the other columns
of $\jac \tilde{H}$, and
$$
\jac (\tilde{H}|_{x_1=1}) = \big(\jac \tilde{H})|_{x_1=1} \cdot (\jac (1,x_2,x_3,\ldots,x_m)\big),
$$
we infer that $\rk \jac (\tilde{H}|_{x_1=1}) = 2$, and we may apply \cite[Theorem 2.3]{bondt17}.
\begin{compactitem}

\item In the case of \cite[Theorem 2.3]{bondt17} (1), case (2) of Theorem
\ref{rkr} follows, which yields (5) of Theorem \ref{rk3}.

\item In the case of \cite[Theorem 2.3]{bondt17} (2), case (1) of Theorem
\ref{rk3} follows.

\item In the case of \cite[Theorem 2.3]{bondt17} (3), case (1) or case (2)
of Theorem \ref{rk3} follows.

\item Case (4) of \cite[Theorem 2.3]{bondt17} cannot occur.

\end{compactitem}

\item $c \ne 0$.

By Lemma \ref{Ccoldep1}, there exists a $v \in K^n$ of which the first $3$
coordinates are not all zero, such that
$
\jac \tilde{H} \cdot v = (x_1,x_2,x_3,0,\ldots,0)\tp.$ Notice that $\jac \tilde{H}_{4}=(x_1, c x_2,   0, 0,\ldots,0\,)$ and the column space of $B$ does not contain a nonzero constant vector.  We deduce from Lemma \ref{rk3calc} that  case (4) of Theorem \ref{rk3} follows.

\end{compactitem}

\item \emph{None of the above.}

We first show that $\rk B \ge 2$. So assume that $\rk B \le 1$.
Since the columns of $B$ are not dependent over $K$ in pairs, we deduce from \cite[Theorem 2.1]{bondt17} that the rows of
$B$ are dependent over $K$ in pairs. This contradicts the fact that the column
space of $B$ does not contain a nonzero constant vector.
So $\rk B \ge 2$ indeed.

From $CB=0$, we have $\rk B + \rk C \leq 3$, and thus $\rk C \le 1$. If $C = 0$, then (1)
of Theorem \ref{rk3} follows. If $\rk C = 1$, then
$\rk B =2$ and $\rk B + \rk C = 3$. From Lemmas \ref{Ccoldep2} and \ref{BCrkr}, we deduce
that the column space of $B$ contains a nonzero constant vector,  a
contradiction.
\end{itemize}

So it remains to prove the last claim. This is trivial in case of (1) of Theorem \ref{rk3}.
In case of (4) of Theorem \ref{rk3}, the last claim follows from the fact that
$\tilde{H}_1^2 + c \tilde{H}_2^2 - 4 \tilde{H}_3 \tilde{H}_4 = 0$. In all other cases,
the last claim follows from Lemma \ref{23} or the last claim of Theorem \ref{rkr}.
\end{proof}

\begin{lemma} \label{F2lem}
Let $K$ be a field of characteristic $2$ and $L$ be an
extension field of $K$. If Theorem \ref{rk3} holds for $L$, then it holds for $K$.
\end{lemma}

\begin{proof}
Suppose that $H\in K[x]^m$ satisfies Theorem \ref{rk3} over $L$, i.e., there exist $S\in \GL_m(K)$ and $T\in \GL_n(K)$ such that $\tilde{H}:=SH(Tx)$ is of the form (1), (3) or (6)
in Theorem \ref{rk3}. We assume that $\tilde{H}$ is of the form (3) because the other cases follows in a similar manner
as Lemma \ref{Fqlem}.

Notice that
$$
\big(\,0 ~ x_3 ~ x_2 ~ x_1 ~ y_4 ~ y_5 ~ \cdots ~ y_m\,\big) \cdot
\jac \tilde{H} = 0.
$$
Since $\jac \tilde{H} = S(\jac H)|_{Tx} T$, one may verify that
$$
\big(\,0 ~ x_3 ~ x_2 ~ x_1 ~ y_4 ~ y_5 ~ \cdots ~ y_m\,\big) \cdot S \cdot \jac H = 0.
$$

Suppose first that $m = 4$. As $\rk \jac H = m-1$, there exists a nonzero
$v \in K(x)^m$ such that $\ker \big(\,v_1 ~ v_2 ~ v_3 ~ v_4\,\big)$ is equal
to the column space of $\jac H$. Since the column space of $\jac H$ is contained
in $\ker \big((\,0~x_3~x_2~x_1\,)S\big)$, it follows that
$\big(\, v_1 ~ v_2 ~ v_3 ~ v_4 \,\big)$ is dependent on $(\,0~x_3~x_2~x_1\,) S$.
So
$$
v_1 (S^{-1})_{11} + v_2 (S^{-1})_{21} + v_3 (S^{-1})_{31} + v_4 (S^{-1})_{41} = 0
$$
and the components of $v$ are dependent over $L$. Consequently, the
components of $v$ are dependent over $K$. So $\ker \big(\,v_1 ~ v_2 ~ v_3 ~ v_4\,)$
contains a nonzero vector over $K$, and so does the column space of $\jac H$.
Now we can follow the same argumentation as in the first case in the proof
of Theorem \ref{rk3}.

Suppose next that $m > 4$. Then the rows of $\jac H$ are dependent over
$L$ and thus dependent over $K$ as well. So we
may assume that the last row of $\jac H$ is zero. By induction on $m$,
$(H_1,H_2,\ldots,H_{m-1})$ is as $H$ in Theorem \ref{rk3}. As $H_m = 0$,
we conclude that $H$ satisfies Theorem \ref{rk3} over $K$.
\end{proof}

\section{Keller maps $x+H$ with $H$ quadratic homogeneous and $\rk \jac  H=3$}

In this section, we classify all Keller maps $x+H$ over an arbitrary field $K$ where $H$ is quadratic homogeneous and $\rk \jac H\leq 3$. Notice that for any homogeneous polynomial map $H\in K[x]^n$, $\det \jac H\in K^*$ if and only if $\jac H$ is nilpotent (cf. \cite[Lemma 6.2.11]{essen2000}).

Recall that a polynomial map $F=x+H\in K[x]^n$ is called triangular if $H_n\in K$ and $H_i\in
K[x_{i+1},\ldots,x_n],$ $1\leq i \leq n-1$. A
polynomial map $F$ is called linearly triangularizable if it is
linearly conjugate to a triangular map, i.e., there exists a $T\in \GL_n(K)$ such that $T^{-1}F(Tx)$ is triangular. A
linearly triangularizable map is a tame automorphism.

\begin{lemma} \label{lem3}
Let $H \in K[x]^3$ be quadratic homogeneous, such that
$\jac_{x_1,x_2,x_3} H$ is nilpotent.
Then $\jac_{x_1,x_2,x_3} H$ is similar over $K$ to a triangular matrix or to
a matrix of the form
$$
\left( \begin{array}{ccc}
0 & f & 0 \\ b & 0 & f \\ 0 & -b & 0
\end{array} \right)
$$
where $f$ and $b$ are independent linear forms in $K[x_4,x_5,\ldots,x_n]$.
\end{lemma}

\begin{proof}
Suppose that $\jac_{x_1,x_2,x_3} H$ is not similar over $K$ to a triangular matrix.
Take $i$ such that the coefficient matrix of $x_i$ of $\jac_{x_1,x_2,x_3} H$
is nonzero, and define
$$
N := \jac_{x_1,x_2,x_3} (H|_{x_i=x_i+1}) = (\jac_{x_1,x_2,x_3} H)|_{x_i=x_i+1.}
$$
Then $N$ is nilpotent, and $N$ is not similar over $K$ to a triangular matrix. Since
$N(0)$ is nilpotent, it is similar over $K$ to $E_{13}$ or $E_{12}+E_{23}$. By \cite[Lemma 3.1]{bondt17} $N$ is similar over $K$ to a matrix of the form
$$
\left( \begin{array}{ccc}
0 & f+1 & 0 \\ b & 0 & f+1 \\ 0 & -b & 0
\end{array} \right),
$$
where $b$ and $f$ are linear forms, and $b$ and $f$ are independent because the
coefficients of $x_i$ in $b$ and $f$ are $0$ and $1$ respectively. So there exists a $T\in \GL_3(K)$ such that
$T^{-1}(\jac_{x_1,x_2,x_3} H)T$ is of the form $$
\left( \begin{array}{ccc}
0 & f & 0 \\ b & 0 & f \\ 0 & -b & 0
\end{array} \right),
$$
where $b$ and $f$ are independent linear forms. Let $\hat{T}=\diag(T, I_{n-3})$ and
$\tilde{H}=T^{-1}H(\hat{T}x)$. Then $$\jac_{x_1,x_2,x_3} \tilde{H}=\left.\left( \begin{array}{ccc}
0 & f& 0 \\ b & 0 & f\\ 0 & -b & 0
\end{array} \right)\right|_{x=\hat{T}x.}$$
The coefficients of $x_1 x_2$ and $x_2 x_3$ in $\tilde{H}_2$ are zero, so
$b(\hat{T}x)$ and $f(\hat{T}x)$ do not contain $x_2$. In $\tilde{H}_1$ and
$\tilde{H}_3$, these coefficients are zero as well, so $b(\hat{T}x)$ and $f(\hat{T}x)$
do not contain $x_1,x_2,x_3$, and neither do $b$ and $f$.
\end{proof}

\begin{lemma} \label{lem1}
Let $H \in K[x]^n$ with $\jac H$  nilpotent. Suppose that (i) $\jac H$ may only be nonzero in the first row and the
first $2$ columns (resp.\@ (ii) $\jac H$ may only be nonzero in the first row and the first $3$ columns with $\chr K=2$). Then there exists a $T \in \GL_n(K)$ such that for
$\tilde{H} := T^{-1} H(Tx)$, the following holds.

\begin{enumerate}[(a)]

\item $\jac \tilde{H}$ may only be nonzero in the first row and the
first $2$ (resp.\@ $3$) columns.

\item The Hessian matrix of the leading part with
respect to $x_2,x_3,\ldots,x_n$ of $\tilde{H}_1$ is the product of
a symmetric permutation matrix and a diagonal matrix.

\item Every principal minor of the leading principal submatrix
of size $2$ (resp.\@ $3$) of $\jac \tilde{H}$ is zero.

\end{enumerate}
\end{lemma}

\begin{proof}
By Proposition \ref{evenrk}, there exists a lower triangular
$T \in \Mat_{n}(K)$, for which the diagonal elements are all 1 and the first column is $e_1$,
such that the Hessian matrix of the leading part with
respect to $x_2,x_3,\ldots,x_n$ of $\tilde{H}_1 = H_1(Tx)$ is the
product of a symmetric permutation matrix and a diagonal matrix.

Furthermore, $\jac \tilde{H}$ may only be nonzero in the first row
and the first $2$ (resp.\@ $3$) columns
because of the form of $T$. So it remains to show (c). We discuss the two cases respectively.

\begin{enumerate}[(i)]
\item
Let $N$ be the leading principal submatrix of size $2$ of $\jac \tilde{H}$.

Suppose first that $\jac \tilde{H}$ may only be nonzero in the first $2$
columns.  Then $N$ is nilpotent since $\jac \tilde{H}$ is nilpotent.
On account of \cite[Theorem 3.2]{bondt17}, $N$ is similar over $K$
to a triangular matrix.
Hence the rows of $N$ are dependent over $K$.
If the second row of $N$ is zero, then (c) follows. If the second row of
$N$ is not zero, then we may assume that the first row of $N$ is zero, and (c)
follows as well.

Suppose next that $\jac \tilde{H}$  may only be nonzero in the first row and
the first $2$ columns, but not just the first $2$ columns. Then $\frac{\partial}{\partial{x_1}} \tilde{H}_2, \frac{\partial}{\partial{x_2}} \tilde{H}_2 \in K[x_1,x_2]$, and $\parder{}{x_1} \tilde{H}_1 \in
K[x_1,x_2]$ as well since $\tr \jac \tilde{H} = 0$. We distinguish two
cases.
\begin{compactitem}

\item $\parder{}{x_2} \tilde{H}_1 \in K[x_1,x_2]$.

 Let $G := \tilde{H}(x_1,x_2,0,\ldots,0)$. Then
$\jac G = (\jac \tilde{H})|_{x_3=\cdots=x_n=0}$. Consequently,
the nonzero part of $\jac G$ is restricted to the first two columns.

So the leading principal submatrix of size $2$ of $\jac G$
is nilpotent. But this submatrix is just $N$, and just as for
$\tilde{H}$ before, we may assume that only one row of $N$ is nonzero.
This gives (c).

\item $\parder{}{x_2} \tilde{H}_1 \notin K[x_1,x_2]$.

Since $\hess (\tilde{H}_1|_{x_1=0})$ is the product of a permutation matrix and a diagonal matrix,
it follows that $\parder{}{x_2} \tilde{H}_1$ is a linear combination of $x_1$ and $x_i$,
where $i \ge 3$, such that $x_i$ does not occur in any other entry of $\jac \tilde{H}$.
Looking at the coefficient of $x_i^1$ in the sum of the
principal minors of size $2$, we infer that $\parder{}{x_1} \tilde{H}_2 = 0$.

So the second row of $\jac \tilde{H}$ is $(\parder{}{x_2}\tilde{H}_2)e_2\tp$, and thus $\parder{}{x_2} \tilde{H}_2$ is zero since $\jac \tilde{H}$ is nilpotent.
Hence the second row of $\jac \tilde{H}$ is zero. Since $\tr \jac \tilde{H} = 0$, we infer (c).

\end{compactitem}

\item
Let $N$ be the principal submatrix of size $3$ of $\jac \tilde{H}$.

Suppose first that $\jac \tilde{H}$ may only be nonzero in the first $3$ columns.
Then $N$ is nilpotent. On account of
\cite[Theorem 3.2]{bondt17}, $N$ is similar over $K$ to a triangular matrix.
But for a triangular nilpotent Jacobian matrix of size $3$ over a field of characteristic
$2$, the rank cannot be $2$. So $\rk N \le 1$.

Hence the rows of $N$ are dependent over $K$ in pairs.
If the second and the third row of $N$ are zero, then (c) follows.
If the second or the third row of $N$ is not zero, then we may assume that
the first $2$ rows of $N$ are zero, and (c) follows as well.

Suppose next that $\jac \tilde{H}$  may only be nonzero in the first row
and the first $3$ columns, but not just the first $3$ columns. We distinguish
three cases.
\begin{compactitem}

\item $\parder{}{x_2} \tilde{H}_1, \parder{}{x_3} \tilde{H}_1 \in K[x_1,x_2,x_3]$.

Using techniques of the proof of (i), we can reduce to the case where
$\jac \tilde{H}$ may only be nonzero in the first $3$ columns.

\item $\parder{}{x_2} \tilde{H}_1, \parder{}{x_3} \tilde{H}_1 \notin K[x_1,x_2,x_3]$.

Using techniques of the proof of (i), we can deduce that
$\parder{}{x_1} \tilde{H}_2 = \parder{}{x_1} \tilde{H}_3 = 0$,
and that $\jac_{x_2,x_3} (\tilde{H}_2,\tilde{H}_3)$ is nilpotent.
On account of \cite[Theorem 3.2]{bondt17},
$\jac_{x_2,x_3} (\tilde{H}_2,\tilde{H}_3)$ is similar over $K$
to a triangular matrix. But a triangular nilpotent Jacobian matrix
of size $2$ over a field of characteristic $2$ must be zero.
So $\jac_{x_2,x_3} (\tilde{H}_2,\tilde{H}_3) = 0$.
Consequently, the last two rows of $N$ are zero, and (c) follows.

\item None of the above.

Assume without loss of generality that $\parder{}{x_2} \tilde{H}_1 \in
K[x_1,x_2,x_3]$ and $\parder{}{x_3} \tilde{H}_1 \notin K[x_1,x_2,x_3]$.
Since $\hess (\tilde{H}_1|_{x_1=0})$ is the product of a permutation matrix and a diagonal
matrix, it follows that $\parder{}{x_3} \tilde{H}_1$ is a linear combination of
$x_1$ and $x_i$, where $i \ge 4$, such that $x_i$ does not occur in any other
entry of $\jac \tilde{H}$.

Looking at the coefficient of $x_i^1$ in the sum of the
principal minors of size $2$, we infer that $\parder{}{x_1} \tilde{H}_3 = 0$. If
$\parder{}{x_1} \tilde{H}_2 = 0$ as well, then we can advance as above, so assume that
$\parder{}{x_1} \tilde{H}_2 \neq 0$.
Looking at the coefficient of $x_i^1$ in the
sum of the principal minors of size $3$, we infer that
$\parder{}{x_2} \tilde{H}_3 = 0$. Then the third row of $\jac \tilde{H}$ is
$(\parder{}{x_2} \tilde{H}_3)e_3\tp$. So
$\parder{}{x_3} \tilde{H}_3=0$ since $\jac\tilde{H}$ is nilpotent. Hence the third row of
$\jac \tilde{H}$ is zero.

From $\tr \jac \tilde{H} = 0$, we deduce that
$\parder{}{x_1} \tilde{H}_1 = -\parder{}{x_2} \tilde{H}_2$. We show that
\begin{equation} \label{DN0}
\parder{}{x_1} \tilde{H}_1 = \parder{}{x_2} \tilde{H}_2 = 0.
\end{equation}
For that purpose, suppose that $\parder{}{x_1} \tilde{H}_1 \neq 0$.
Since  $\parder{}{x_1} x_1^2 = 0$, the coefficient of $x_1$ in
$\parder{}{x_1} \tilde{H}_1$ is zero. Similarly, the coefficient of $x_2$ in
$\parder{}{x_2} \tilde{H}_2$ is zero. As $\tilde{H}_2 \in K[x_1,x_2,x_3]$, we
infer that
$
\parder{}{x_1} \tilde{H}_1 = -\parder{}{x_2} \tilde{H}_2 \in K x_3 \setminus \{0\}.
$

Looking at the
coefficient of $x_3^2$ in the sum of the
principal minors of size $2$, we deduce that the coefficient of $x_3^2$ in
$(\parder{}{x_i} \tilde{H}_1) \cdot (\parder{}{x_1} \tilde{H}_i)$ is nonzero.
Consequently, the coefficient of $x_3^3$ in
$(\parder{}{x_i} \tilde{H}_1) \cdot (\parder{}{x_2} \tilde{H}_2) \cdot
(\parder{}{x_1} \tilde{H}_i) \in K x_3^3 \setminus \{0\}$ is nonzero.
This contributes to the coefficient of $x_3^3$ in the sum of the principal minors of size $3$,  a contradiction because this contribution
cannot be canceled.

So \eqref{DN0} is satisfied. We show that in addition,
\begin{equation} \label{N120}
\parder{}{x_2} \tilde{H}_1 = 0.
\end{equation}
The coefficient of $x_1$ of $\parder{}{x_2} \tilde{H}_1$ is zero, because of \eqref{DN0}.
The coefficient of $x_2$ of $\parder{}{x_2} \tilde{H}_1$ is zero since
$\parder{}{x_2} x_2^2 = 0$.
The coefficient of $x_3$ of $\parder{}{x_2} \tilde{H}_1$ is zero, because
the coefficient of $x_2$ of $\parder{}{x_3} \tilde{H}_1 \in K x_1 + K x_i$ is zero.
So \eqref{N120} is satisfied as well.

Recall that the third row of $N$ is zero.
From \eqref{DN0} and \eqref{N120}, it follows that the diagonal and the second column
of $N$ are zero as well. Hence every principal minor of $N$ is zero, which gives (c).
\qedhere

\end{compactitem}
\end{enumerate}
\end{proof}

\begin{lemma} \label{lem2}
Let $\tilde{H}$ be as in Lemma \ref{lem1}. Suppose that $\jac \tilde{H}$
has a principal submatrix $M$ of which the determinant is nonzero. Then
\begin{enumerate}[(1)]

\item $\tilde{H}$ is as in (ii) of Lemma \ref{lem1};

\item rows $2$ and $3$ of $\jac \tilde{H}$ are zero;

\item $M$ has size $2$ and $x_2 x_3 \mid \det M$;

\item Besides $M$, there exists exactly one principal minor matrix $M'$
of size $2$ of $\jac H$, such that $\det M' = - \det M$.

\end{enumerate}
\end{lemma}

\begin{proof}
Take for $N$ the leading principal submatrix of size $2$ (resp.\@ $3$)
of $\jac \tilde{H}$.
Then $M$ is not a principal minor matrix of $N$. So if $M$ does not
contain the upper left corner of $\jac \tilde{H}$, then the last column
of $M$ is zero. Hence $M$ does contain the upper left corner of
$\jac \tilde{H}$.

If $M$ has two columns outside the column range of $N$,
then both columns are dependent on $e_1$. So $M$ has exactly one column
outside the column range of $N$, say column $i$.
\begin{enumerate}[(i)]

\item Suppose first that $\tilde{H}$ is as in (i) of Lemma \ref{lem1}.
Then either $M$ has size $2$ with row and column indices $1$ and $i$,
or $M$ has size $3$ with row and column indices $1$, $2$ and $i$.

The coefficient of $x_1$ in the upper right corner of $M$ is zero, because
$N_{11} = 0$. Hence the upper right corner
of $M$ is of the form $c x_j$ for some nonzero $c \in K$ and a $j \ge 2$.
If $j \ge 3$, then $x_j$ does not appear in any other position of $\jac \tilde{H}$,
and thus all other minors of the same size of $M$ contains no $x_j$, contradicting
the nilpotency of $\jac\tilde{H}$.

So $j = 2$. Now $\det M$ is the only nonzero principal minor of its size which
belongs to $K[x_1,x_2]$, contradicting the nilpotency of $\jac\tilde{H}$ as well.

\item Suppose next that $\tilde{H}$ is as in (ii) of Lemma \ref{lem1}.
If the second row of $\jac\tilde{H}$ is nonzero, then the coefficient of
$x_1 x_3$ in $\tilde{H}_2$ is nonzero, because $N_{22} = 0$.
If the third row of $\jac\tilde{H}$
is nonzero, then the coefficient of $x_1 x_2$ in $\tilde{H}_3$ is nonzero,
because $N_{33} = 0$.
Since every principal minor of $N$ is zero, we infer that
$N_{23} N_{32} = 0$, so either the second or the third row of $\jac\tilde{H}$ is zero.

Assume without loss of generality that the second row of $\jac \tilde{H}$
is zero. Then either $M$ has size $2$ with row and column indices $1$ and $i$,
or $M$ has size $3$ with row and column indices $1$, $3$ and $i$.
The upper right corner of $M$ is of the form $c x_j$ for some nonzero
$c \in K$, and with the techniques in (i) above, we see that $2 \le j \le 3$.

Furthermore, we infer with the techniques in (i) above that $\jac \tilde{H}$
has another principal submatrix $M'$ of the same size as $M'$, of which the
determinant is nonzero as well. The upper right corner of $M'$ can only be
of the form $c' x_{5-j}$ for some nonzero $c' \in K$.

It follows that $N_{12} \ne 0$ and $N_{13} \ne 0$. Consequently, $N_{21} = N_{31} = 0$.
This is only possible if both the second and the third
row of $\jac \tilde{H}$ are zero. So $M$ has size $2$, and claims (3) and (4)
follow. \qedhere

\end{enumerate}
\end{proof}

\begin{lemma} \label{lem4} Let $\tilde{H}=\big(x_1 x_3 + c x_2 x_4, x_2 x_3 - x_1 x_4, \frac12 x_3^2 + \frac{c}2 x_4^2, \frac12 x_1^2 + \frac{c}2 x_2^2\big)$ as in Lemma \ref{rk3calc}, where $c \neq 0$. Let $M \in \Mat_{4}(K)$ be such that $\deg \det \big(\jac{\tilde{H}} + M\big) \le 2$.
Then there exists a translation $G$, such that
$$
\tilde{H}\big(G(x)\big) - \big(\tilde{H} + Mx\big) \in K^4.
$$
In particular, $\det \big(\jac{\tilde{H}} + M\big) = \det \jac \big(\tilde{H} + Mx\big) = 0$.
\end{lemma}

\begin{proof}
Since the quartic part of $\det (\jac{\tilde{H}} + M)$ is zero, we deduce that
$\det (\jac{\tilde{H}}) = 0$.
By way of completing the squares, we can choose a translation $G$ such that
the linear part of $F := \tilde{H}\big(G^{-1}(x)\big) + M\,G^{-1}(x)$ is of the form
$$(a_1 x_1 + b_1 x_2 + c_1 x_3 + d_1 x_4,~a_2 x_1 + b_2 x_2 + c_2 x_3 + d_2 x_4,~a_3 x_1 + b_3 x_2,~c_4 x_3 +d_4 x_4).$$

Notice that $\deg \det \jac F \le 2$. Looking at the coefficients of
$x_1^3$, $x_2^3$, $x_3^3$, and $x_4^3$ of $\det \jac F$, we see that
$b_3 = a_3 = d_4 = c_4 = 0$. Looking at the coefficients of
$x_1^2 x_3$, $x_1 x_3^2$, $x_2^2 x_4$, and $x_2 x_4^2$ of $\det \jac F$,
we see that $b_1 = d_1 = a_1 = c_1 = 0$. Looking at the coefficients of
$x_1^2 x_4$, $x_1 x_4^2$, $x_2^2 x_3$, and $x_2 x_3^2$ of $\det \jac F$,
we see that $b_2 = c_2 = a_2 = d_2 = 0$.

So $F$ has trivial linear part, and $\tilde{H} - F \in K^4$.
Hence $\tilde{H}(G) - F(G) \in K^4$, as claimed. The last claim follows from
$\det (\jac{\tilde{H}}) = 0$.
\end{proof}

\begin{theorem} \label{rk3np}
Let $x+H \in K[x]^n$ be a Keller map with $H$ quadratic homogeneous and $\rk \jac H \le 3$. Then
$x+H$ (up to a square part if $\chr K=2$) is linearly conjugate to one of the following automorphisms:
\begin{enumerate} [(1)]
\item a triangular automorphism;

\item $(x_1+x_2x_5+u_1,~x_2+x_1x_4-x_3x_5+u_2,~x_3+x_2x_4+u_3,~x_4,\ldots,x_n),$ where $u_1,u_2,u_3\in K[x_4,x_5,\ldots,x_n]$;

\item $(x_1+x_2x_6,~x_2+x_1x_5-x_3x_6+ax_4x_5-bx_4x_6+u_2,~x_3+x_2x_5,~x_4+x_5x_6,~x_5,\ldots,x_n)$ with $\chr K=2$, where $u_2\in K[x_4,x_7,x_8,\ldots,x_n]$.
\end{enumerate}
In  particular, $x + H$ is a tame automorphism (up to a square part if $\chr K =2$).
\end{theorem}

\begin{proof} 
Note first that $\jac H$ is nilpotent since $x+H$ is a Keller map and $H$ is homogeneous.
By \cite[Theorem 3.2]{bondt17}, if
$\rk \jac H \le 2$ then $\jac H$ is similar over $K$ to a triangular matrix, whence $x+H$ is linearly triangularizable and thus tame (up to a square part if $\chr K=2$).

So assume that $\rk \jac H = 3$. We follow the cases
of Theorem \ref{rk3}.
\begin{itemize}

\item $H$ is as in (1) of Theorem \ref{rk3}.

Let $\tilde{H} = S H(S^{-1}x)$. Then only the first $3$ rows of
$\jac \tilde{H}$ may be nonzero. If the leading principal submatrix $N$ of size $3$ of $\jac \tilde{H}$ is similar over $K$
to a triangular matrix, then so is $\jac \tilde{H}$ itself.
So assume that $N$ is not similar over $K$ to a triangular matrix.
Then by Lemma \ref{lem3}, $N$ is similar over $K$ to
a matrix of the form
$$
\left( \begin{array}{ccc}
0 & f & 0 \\ b & 0 & -f \\ 0 & b & 0
\end{array} \right),$$
where $f$ and $b$ are independent linear forms in $K[x_4,x_5,\ldots,x_n]$. Replacing $\tilde{H}$ by
$T\tilde{H}(T^{-1}x)$ for some appropriate $T\in \GL_n(K)$, we may assume that $N$ is of the form
$$
\left( \begin{array}{ccc}
0 & x_5 & 0 \\ x_4 & 0 & -x_5 \\ 0 & x_4 & 0
\end{array} \right).
$$
So when $\chr K\neq 2$, $x+\tilde{H}$ is of the form as in (2) $$(x_1+x_2x_5+u_1,~x_2+x_1x_4-x_3x_5+u_2,~x_3+x_2x_4+u_3,~x_4,\ldots,x_n),$$ where $u_1,u_2,u_3\in K[x_4,x_5,\ldots,x_n]$. Denote by $E_{i,a}$ the elementary automorphism $(x_1,\ldots,x_{i-1},x_i+a,x_{i+1},\ldots,x_n)$.
Then $x+\tilde{H}=E_{1,u_1}\circ E_{2,u_2}\circ E_{3,u_3}\circ E_{2,x_1x_4-x_3x_5}\circ E_{3,x_2x_4}\circ E_{1,x_2x_5}$, and thus  $x+\tilde{H}$ is tame.
And when $\chr K= 2$, the square-free part of $x+\tilde{H}$ is of that form which is tame.

\item $H$ is as in (2) of Theorem \ref{rk3}.

Let $\tilde{H} = T^{-1} H(Tx)$. Then the rows of $\jac_{x_3,x_4,\ldots,x_n} \tilde{H}$
are dependent over $K$ in pairs. Suppose first that the first $2$ rows of
$\jac_{x_3,x_4,\ldots,x_n} \tilde{H}$ are zero. Then we may assume that
only the last row of $\jac_{x_3,x_4,\ldots,x_n} \tilde{H}$ may be nonzero.

Then the
leading principal submatrix $N$ of size $2$ of $\jac \tilde{H}$ is
nilpotent since $\jac\tilde{H}$ is nilpotent. On account of \cite[Theorem 3.2]{bondt17}, $N$ is
similar over $K$ to a triangular matrix. And we deduce that $\jac \tilde{H}$ is similar over $K$ to a triangular matrix. So
we may choose $T$ such that $\jac \tilde{H}$ is lower triangular, and
(1) is satisfied.

Suppose next that the first $2$ rows of
$\jac_{x_3,x_4,\ldots,x_n} \tilde{H}$ are not both zero. Then we may choose
$T$ such that only the first row of $\jac_{x_3,x_4,\ldots,x_n} \tilde{H}$ may
be nonzero. From Lemmas \ref{lem1} and \ref{lem2}, we
may choose $T$ such that every principal minor of $\jac \tilde{H}$ is zero. From \cite[Lemma 1.2]{dru}, it follows that $\jac \tilde{H}$ is
permutation similar to a triangular matrix, and thus (1) is satisfied.

\item $H$ is as in (3) of Theorem \ref{rk3}.

Let $\tilde{H} = T^{-1} H(Tx)$. Then the rows of $\jac_{x_4,x_5,\ldots,x_n} \tilde{H}$
are dependent over $K$ in pairs. Suppose first that the first $3$ rows of
$\jac_{x_4,x_5,\ldots,x_n} \tilde{H}$ are zero. Then we may choose $T$ such that
only the last row of $\jac_{x_4,x_5,\ldots,x_n} \tilde{H}$ may be nonzero,
and just as above, (1) is satisfied.

Suppose next that the first $3$ rows of
$\jac_{x_4,x_5,\ldots,x_n} \tilde{H}$ are not all zero. Then we may choose
$T$ such that only the first row of $\jac_{x_4,x_5,\ldots,x_n} \tilde{H}$ may
be nonzero. If we can choose $T$ such that every principal minor of
$\jac \tilde{H}$ is zero, then (1) is satisfied, just as before.

So assume that we cannot choose $T$ such that every principal minor of
$\jac \tilde{H}$ is zero. By Lemma \ref{lem1} and
Lemma \ref{lem2}, we may choose $T$ such that $\tilde{H}$
is as in Lemma \ref{lem2}. More precisely, we may choose $T$ such that
$\jac \tilde{H}$ is of the form
\begin{equation}\label{eq2}\left(
  \begin{array}{c|c}
          \begin{array}{ccc}
        0 & x_4 & -x_5 \\
        0 & 0 & 0 \\
        0 & 0 & 0 \\
        x_3 & ax_3 & x_1+ax_2 \\
        x_2 & x_1+bx_3 & bx_2 \\
      \end{array}
        & \begin{array}{cccc}
            x_2 & -x_3 & * & \cdots \\
            0 & 0 & 0 & \cdots \\
            0 & 0 & 0 & \cdots \\
            0 & 0 & 0 & \cdots \\
            0 & 0 & 0 & \cdots
          \end{array}
         \\ \hline
    M & 0 \\
  \end{array}
\right).
\end{equation}
If $M=0$ then $H$ is as in (1)
of Theorem \ref{rk3}, which is the first case. So assume that $M\neq 0$.
Since $\parder{}{x_1} x_1^2 = 0$, the coefficients of $x_1$
in the first column of $M$ are zero.
Hence we can clean the first column of $M$ by way of row operations in \eqref{eq2}
with rows $4$ and $5$, and furthermore by way of a linear
conjugation, because if an element in the first column of $M$ is nonzero, then
the transposed entry in the first row of \eqref{eq2} is zero, so the corresponding
column operations will not have any effect.

Then each row of $M$ is of the form $(0,cx_3,cx_2)$, and thus by linear conjugation, we may assume that the first row of $M$ is $(0,x_3,x_2)$ and all the other rows of $M$ are zero. Furthermore, if we take $$S=(x_1,x_2,x_3,x_4+ax_6,x_5+bx_6,x_6,x_7,\ldots,x_n),$$ then $\jac (S^{-1}\circ \tilde{H} \circ S)$ is of the form as in \eqref{eq2} with $a=b=0$ and replacing $x_4$ by $x_4+ax_6$ and replacing $x_5$ by $x_5+bx_6$.
Then the square-free part of $S^{-1}\circ\tilde{H}\circ S$ is of the form $$\big(x_2(x_4+ax_6)-x_3(x_5+bx_6)+u_1,~0,~0,~x_1x_3,~x_1x_2,~x_2x_3,~0,\ldots,0
\big),$$ where $u_1\in K[x_6,x_7,\ldots,x_n]$. Let $$P=(x_2,x_5,x_6,x_1,x_3,x_4,x_7,\ldots,x_n).$$ Then $P^{-1}=(x_4,x_1,x_5,x_6,x_2,x_3,x_7,\ldots,x_n),$
and one may verify that the square-free part of $P^{-1} \circ S^{-1}\circ (x+\tilde{H})\circ S\circ P$ is of the form as in (3)
$$x+(x_2x_6,~x_1x_5-x_3x_6+ax_4x_5-bx_4x_6+u_2,~x_2x_5,~x_5x_6,~0,\ldots,0),$$
where $u_2\in K[x_4,x_7,x_8,\ldots,x_n]$,
which is equal to $$E_{4,x_5x_6}\circ E_{2,x_1x_5-x_3x_6+ax_4x_5-bx_4x_6+u_2}\circ E_{1,x_2x_6}\circ E_{3,x_2x_5}$$ and thus tame.

\item $H$ is as in (4) of Theorem \ref{rk3}.

Then only the first $4$ columns of $\tilde{H} = T^{-1} H(Tx)$ may be nonzero.
Hence the leading principal submatrix $N$ of size $4$ of $\jac \tilde{H}$
is nilpotent.

Suppose that the rows of $N$ are linearly independent over $K$.
Then there exists an $U \in \GL_4(K)$, such that $U N$ is as $\jac \tilde{H}$
in Lemma \ref{lem4}. Furthermore,
$$
\det (UN + U) = \det U \det (N + I_4) = \det U \in K^{*}.
$$
So $\det (UN + U) \ne 0$ and $\deg \det (UN + U) \le 2$,
contradicting Lemma \ref{lem4}.

So the rows of $N$ are linearly dependent over $K$.
Then the first case of the proof applies for the map
$(\tilde{H}_1,\tilde{H}_2,\tilde{H}_3,\tilde{H}_4)$.
Since $\tilde{H}_i\in K[x_1,x_2,x_3,x_4]$, $1\leq i\leq 4$, the case where $N$ is not similar
over $K$ to a triangular matrix cannot occur as in the first case of the proof. So
$N$ is similar over $K$ to a triangular matrix, and so are $\jac \tilde{H}$
and $\jac H$, and thus (1) is satisfied.

\item  $H$ is as in (5) of Theorem \ref{rk3}.

Let $\tilde{H} = T^{-1} H(Tx)$. Then we have $(\tilde{H}_1,\tilde{H}_2,\tilde{H}_3)\in K[x_1,x_2,x_3]^3$
and $\jac_{x_1,x_2,x_3} (\tilde{H}_1,\tilde{H}_2,\tilde{H}_3)$ is nilpotent.
Then by Theorem \cite[Theorem 3.2]{bondt17},
$\jac_{x_1,x_2,x_3} (\tilde{H}_1,\tilde{H}_2,\tilde{H}_3)$ is
similar over $K$ to a triangular matrix, and so is  $\jac \tilde{H}$.
Then (1) is satisfied.

\item  $H$ is as in (6) of Theorem \ref{rk3}.

Let $\tilde{H} = T^{-1} H(Tx)$. Then we have
$(\tilde{H}_1,\tilde{H}_2,\tilde{H}_3,\tilde{H}_4)\in K[x_1,x_2,x_3,x_4,\allowbreak
x_5^2,x_6^2,\ldots,x_n^2]^4$  and
\begin{equation} \label{xeigen}
\jac_{x_1,x_2,x_3,x_4} (\tilde{H}_1,\tilde{H}_2,\tilde{H}_3,\tilde{H}_4) \cdot
\left( \begin{smallmatrix} x_1 \\ x_2 \\ x_3 \\ x_4 \end{smallmatrix} \right) = 0.
\end{equation}
Furthermore, $\jac_{x_1,x_2,x_3,x_4} (\tilde{H}_1,\tilde{H}_2,\tilde{H}_3,\tilde{H}_4)$ is nilpotent.
If $\rk \jac_{x_1,x_2,x_3,x_4} \allowbreak (\tilde{H}_1,\tilde{H}_2,\tilde{H}_3,\tilde{H}_4) \le 2$,
then (1) is satisfied just as in the previous case.

So assume that
$\rk \jac_{x_1,x_2,x_3,x_4} (\tilde{H}_1,\tilde{H}_2,\tilde{H}_3,\tilde{H}_4) = 3$, whence
its Jordan Normal Form has only one block, so $\big(\jac_{x_1,x_2,x_3,x_4} \allowbreak
(\tilde{H}_1,\tilde{H}_2,\tilde{H}_3,\tilde{H}_4)\big)^3 \ne 0$.
From the proof of \cite[Lemma 2.10]{sun}, we infer that
$$
\big(\jac_{x_1,x_2,x_3,x_4} (\tilde{H}_1,\tilde{H}_2,\tilde{H}_3,\tilde{H}_4)\big)^3 \cdot
\left( \begin{smallmatrix} x_1 \\ x_2 \\ x_3 \\ x_4 \end{smallmatrix} \right) \ne 0
$$
which contradicts \eqref{xeigen}. \qedhere

\end{itemize}
\end{proof}

\paragraph{Acknowledgments}
The first author has been supported by the Netherlands Organisation of Scientific Research (NWO).
The second author has been partially supported by the NSF of China (11771176) and by the China Scholarship Council.



\begin{thebibliography}{99}

\bibitem{bass1982} H. Bass, E. Connel, D. Wright, The Jacobian Conjecture: Reduction of degree and formal expansion of
the inverse, Bull. Amer. Math. Soc., 7 (1982) 287-330.

\bibitem{bondt09} M. de Bondt, Homogeneous Keller maps, Ph.D. thesis, Univ.
of Nij\-me\-gen, The Netherlands, 2009.

\bibitem{bondt13} M. de Bondt, Mathieu subspaces of codimension less than $n$ of $\Mat_n(K)$, arXiv:1310.7843, [math.AC], 2013.

\bibitem{bondt17} M. de Bondt, Quadratic polynomial maps with Jacobian rank two, arXiv: 1061.00579v4 [math.AC], 2017.

\bibitem{bondt172} M. de Bondt, Computations of keller maps over fields with $\frac{1}{6}$, arXiv:1601.09753, [math.AC], 2017.

\bibitem{dru} L. Dru\.zkowski, The Jacobian conjecture in case of rank or corank less than three, J. Pure Appl. Algebra, 85(3)(1993) 233--244.

\bibitem{essen2000} A. van den Essen, Polynomial automorphisms and the Jacobian Conjecture, Progress in Mathematics, Vol. 190, Birkh\"{a}user,
Basel-Boston-Berlin, 2000.

\bibitem{jung} H. Jung, \"{U}ber ganze birationale Transformationen der
Ebene, J. Reine. Angew. Math., 184 (1942) 161--174.

\bibitem{kulk} W. van der Kulk, On polynomial rings in two variables, Nieuw
Archief voor Wiskunde, 3(1)(1953) 33--41.

\bibitem{meisters1991} G. Meisters, C. Olech, Strong nilpotence holds in dimension up to five only, Linear
Multilinear Algebra, 30 (4) (1991) 231-255.

\bibitem{rusek} K. Rusek, Polynomial automorphisms, preprint 456,
Inst. of Math., Polish Acad. of Sciences, IMPAN, \'{S}niadeckich 8,
P.O. Box 137, 00-950 Warsaw, Poland, May 1989.

\bibitem{shestakov2} I. P. Shestakov, U. U. Umirbaev, The tame and the
wild automorphisms of polynomial rings in three variables, J. Amer.
Math. Soc., 17(1)(2004) 197--227.

\bibitem{sun10} X. Sun, On quadratic homogeneous quasi-translations, J. Pure Appl. Algebra, 214(11)(2010) 1962--1972.

\bibitem{sun} X. Sun, Classification of quadratic homogeneous automorphisms in dimension five, Comm. Algebra, 42(7)(2014) 2821--2840.

\bibitem{wang} S. Wang, A Jacobian criterion for separability, J. Algebra,
65(2)(1980) 453-494.
\end{thebibliography}
\end{document}